\documentclass[12pt, reqno]{amsart}
\usepackage{amssymb, mathrsfs}
\usepackage{latexsym}
\usepackage{amsmath}
\usepackage{amsthm}
\usepackage{amsfonts}
\usepackage{dsfont, upgreek}
\usepackage{mathtools}
\usepackage{epsfig}
\usepackage{amscd}
\usepackage{graphicx}
\usepackage{tikz-cd}
\usepackage{enumitem}
\setlist[enumerate]{itemsep=0.5ex}

\usepackage{stmaryrd}

\usepackage{color}
\usepackage{tikz}
\usetikzlibrary{patterns}

\usepackage[left=1.4in,right=1.4in,top=1.2in,bottom=1.2in]{geometry}

\usepackage{tikz}
\usetikzlibrary{decorations.markings}
\usetikzlibrary{arrows.meta}

\usepackage{extarrows}

\usepackage{adjustbox}
\usepackage{pgfplots}
\usepgfplotslibrary{colormaps}
\usepackage{slashed}
\usepackage[nocompress]{cite}

\usepackage[colorlinks=true,linkcolor=blue,citecolor=blue]{hyperref}

\usepackage[colorinlistoftodos,prependcaption,textsize=tiny]{todonotes}  

\usepackage[all,cmtip]{xy}
\theoremstyle{plain}
\newtheorem{theorem}{Theorem}[section]

\newtheorem{question}[theorem]{Question}

\theoremstyle{definition} 
\newtheorem{definition}[theorem]{Definition}
\newtheorem{example}[theorem]{Example}
\newtheorem*{claim*}{Claim}

\theoremstyle{remark}

\numberwithin{equation}{section}

\newcommand{\Sc}{\mathrm{Sc}}

\newcommand{\Bigwedge}{\mathord{\adjustbox{raise=.4ex, totalheight=.7\baselineskip}{$\bigwedge$}}}

\newcommand{\ind}{\textup{Ind}}

\newcommand{\R}{\mathbb{R}}
\newcommand{\tr}{\mathrm{tr}}

\newcommand{\Z}{\mathbb{Z}}

\newcommand{\sph}{\mathbb{S}}

\newcommand{\interior}[1]{%
	{\kern0pt#1}^{\mathrm{\,o}}%
}

\makeatletter
\let\save@mathaccent\mathaccent
\newcommand*\if@single[3]{%
	\setbox0\hbox{${\mathaccent"0362{#1}}^H$}%
	\setbox2\hbox{${\mathaccent"0362{\kern0pt#1}}^H$}%
	\ifdim\ht0=\ht2 #3\else #2\fi
}
\newcommand*\rel@kern[1]{\kern#1\dimexpr\macc@kerna}
\newcommand*\wideaccent[2]{\@ifnextchar^{{\wide@accent{#1}{#2}{0}}}{\wide@accent{#1}{#2}{1}}}
\newcommand*\wide@accent[3]{\if@single{#2}{\wide@accent@{#1}{#2}{#3}{1}}{\wide@accent@{#1}{#2}{#3}{2}}}
\newcommand*\wide@accent@[4]{%
	\begingroup
	\def\mathaccent##1##2{%
		\let\mathaccent\save@mathaccent
		\if#42 \let\macc@nucleus\first@char \fi
		\setbox\z@\hbox{$\macc@style{\macc@nucleus}_{}$}%
		\setbox\tw@\hbox{$\macc@style{\macc@nucleus}{}_{}$}%
		\dimen@\wd\tw@
		\advance\dimen@-\wd\z@
		\divide\dimen@ 3
		\@tempdima\wd\tw@
		\advance\@tempdima-\scriptspace
		\divide\@tempdima 10
		\advance\dimen@-\@tempdima
		\ifdim\dimen@>\z@ \dimen@0pt\fi
		\rel@kern{0.6}\kern-\dimen@
		\if#41
		#1{\rel@kern{-0.6}\kern\dimen@\macc@nucleus\rel@kern{0.4}\kern\dimen@}%
		\advance\dimen@0.4\dimexpr\macc@kerna
		\let\final@kern#3%
		\ifdim\dimen@<\z@ \let\final@kern1\fi
		\if\final@kern1 \kern-\dimen@\fi
		\else
		#1{\rel@kern{-0.6}\kern\dimen@#2}%
		\fi
	}%
	\macc@depth\@ne
	\let\math@bgroup\@empty \let\math@egroup\macc@set@skewchar
	\mathsurround\z@ \frozen@everymath{\mathgroup\macc@group\relax}%
	\macc@set@skewchar\relax
	\let\mathaccentV\macc@nested@a
	\if#41
	\macc@nested@a\relax111{#2}%
	\else
	\def\gobble@till@marker##1\endmarker{}%
	\futurelet\first@char\gobble@till@marker#2\endmarker
	\ifcat\noexpand\first@char A\else
	\def\first@char{}%
	\fi
	\macc@nested@a\relax111{\first@char}%
	\fi
	\endgroup
}
\makeatother

\newcommand\overbar{\wideaccent\overline}

\makeatletter
\newcommand*{\transpose}{%
	{\mathpalette\@transpose{}}%
}
\newcommand*{\@transpose}[2]{%
	\raisebox{\depth}{$\m@th#1\intercal$}%
}
\makeatother

\begin{document}
		\title{Scalar-mean rigidity   beyond warped product spaces}
	
		\author{Jinmin Wang}
		\address[Jinmin Wang]{State Key Laboratory of Mathematical Sciences, Academy of Mathematics and Systems Science, Chinese Academy of Sciences}
		\email{jinmin@amss.ac.cn}
		\thanks{The first author is partially supported by NSFC 12501169.}
	\author{Zhizhang Xie}
\address[Zhizhang Xie]{ Department of Mathematics, Texas A\&M University }
\email{xie@tamu.edu}
\thanks{The second author is partially supported by NSF 2247322.}
	
		\begin{abstract}
			
		The main scalar-mean extremality and rigidity results in the existing literature concern manifolds whose curvature operators are nonnegative, or warped product spaces with a log-concave warping function whose leaves carry metrics of nonnegative curvature operator.	In this paper, we establish scalar-mean extremality and rigidity theorems for a broad class of Riemannian manifolds with boundary whose metrics are conformal to ones with nonnegative curvature operator.  In particular, our results  extend these theorems beyond the warped product setting and yields new families of manifolds exhibiting  scalar-mean extremality and rigidity.
		
		\end{abstract}
		\maketitle

\section{Introduction}

The study of scalar curvature geometry on Riemannian manifolds has been a central theme in differential geometry and topology over the past several decades. Classical results of Schoen–Yau \cite{MR541332} and Gromov–Lawson \cite{MR569070} showed that the $n$-dimensional torus $\mathbb T^n$ admits no Riemannian metric of positive scalar curvature. This fundamental theorem revealed that the existence of positive scalar curvature metrics is subject to deep topological obstructions.

A complementary rigidity phenomenon was discovered by Llarull \cite{MR2636963}, who proved that the standard round sphere is extremal among all manifolds with the same scalar curvature lower bound and distance lower bound. More precisely, if a closed spin manifold $(M^n,g)$ admits a distance–nonincreasing map of nonzero degree to the round sphere $(\mathbb S^n,g_0)$ and its scalar curvature satisfies $\Sc_g \ge n(n-1)$, then $(M,g)$ must be isometric to $(\mathbb S^n,g_0)$.

Since Llarull’s work, numerous generalizations and extensions of his theorem have been developed. Goette and Semmelmann proved a scalar curvature extremality and rigidity theorem for closed manifolds with nonnegative curvature operator and nonvanishing Euler characteristic \cite{GoetteSemmelmann}, and closed K\"ahler manifolds with non-negative Ricci curvature \cite{GoetteSemmelmannKahler}, in the spin setting.

Subsequently, Lott \cite{Lottboundary} extended this line of ideas to manifolds with boundary, establishing scalar–mean extremality and rigidity in the spin setting for manifolds with nonnegative curvature operator, nonnegative boundary second fundamental form, and nonzero Euler characteristic.

To be precise,  we consider the following notion of scalar–mean extremality and rigidity.

\begin{definition}\label{def:extremal}
		Let $(N^n,\partial N,g)$ be a compact Riemannian manifold with boundary. We say that $(N,\partial N,g)$ is \emph{scalar–mean extremal} if, for any Riemannian metric $g_1$ on $N$ such that $g_1\geq g$,  $\Sc_{g_1}\ge \Sc_{g}$ and $H_{g_1}\ge H_{g}$, one has equalities $\Sc_{g_1}=\Sc_{g}$ and $H_{g_1}=H_{g}$. If, in addition, $g_1$ and $g$ must be equal, then $(N,\partial N,g_N)$ is called \emph{scalar–mean rigid}.
\end{definition}

Here $H_{g}$ is the mean curvature of $\partial N$ under the metric $g$. Our convention is that the mean curvature of the standard unit round sphere $\sph^{n-1}$ as the boundary the Euclidean unit ball has mean curvature $n-1$ at every point. 

In this paper, we consider the following notion of scalar–mean extremality and rigidity in the spin setting.

\begin{definition}\label{def:spinextremal}
	Let $(N^n,\partial N,g_N)$ be a compact Riemannian manifold with boundary. We say that $(N,\partial N,g_N)$ is \emph{spin scalar–mean extremal} if, for any compact Riemannian manifold $(M^n,\partial M,g_M)$ admitting a smooth distance non-increasing spin map $f\colon M\to N$ of nonzero degree and satisfying $\Sc_{g_M}\ge f^*\Sc_{g_N}$ and $H_{g_M}\ge f^*H_{g_N}$, one has equalities $\Sc_{g_M}=f^*\Sc_{g_N}$ and $H_{g_M}=f^*H_{g_N}$. If, in addition, $f$ must be a local isometry, then $(N,\partial N,g_N)$ is called \emph{spin scalar–mean rigid}.
\end{definition}

Here, a map $f\colon M\to N$ is called \emph{spin} if $TM\oplus f^\ast TN$ admits a spin structure; equivalently, if the second Stiefel–Whitney classes satisfy $w_2(M) = f^\ast (w_2(N))$. In particular, the identity map is automatically spin.

\medskip

In \cite[Section~5.5]{Gromov4lectures2019}, Gromov proposed studying scalar–mean extremality and rigidity for Riemannian bands modeled by the log–concave warped product metric
\[
g = dr^2 + \varphi(r)^2 h,
\]
where $h$ is, for instance, a flat metric on a torus or the standard metric on a sphere, and the warping function $\varphi$ is log–concave, i.e.\ $(\log\varphi)'' \le 0$. This family of metrics includes many classical examples, such as hyperbolic metrics.

To establish scalar–mean extremality and rigidity in this setting, Gromov proposed analyzing the so–called $\mu$–bubbles associated with the warping function $\varphi$. Inspired by this approach, Cecchini and Zeidler \cite{Cecchini:2021vs} employed the Dirac operator method to prove scalar–mean extremality and rigidity for a class of log–concave warped products in the spin setting. Gromov further conjectured that the same type of rigidity should persist for warped product spaces that are \emph{degenerate} at the ends. Since such degenerate warped products are noncompact and incomplete, the standard index–theoretic arguments no longer apply. By incorporating a Poincaré–type inequality, the authors of \cite{WangXiedegnerate} established scalar curvature rigidity for a broad class of degenerate log–concave warped products in the spin case. In particular, as a special case, they proved scalar curvature rigidity for the round sphere with two antipodal points removed in the spin setting. This special case was also obtained independently by Bär, Brendle, Hanke, and Wang \cite{MR4733718}.

Moreover, scalar–mean extremality and rigidity have also been established for certain codimension–zero submanifolds with (polyhedral) boundary inside log–concave warped products, including parabolic polyhedra in hyperbolic space and radially convex regions in log–concave warped products \cite{Lihyperbolic, WXDihedralWarped, ChaiWanWarped}.

In general, Gromov posed the following question \cite{MR3822551}:
\begin{quote}
	\emph{Find verifiable criteria for extremality and rigidity; decide which manifolds admit extremal or rigid metrics, and describe particular classes of extremal or rigid manifolds.}
\end{quote}

\medskip

In this paper, we show that scalar–mean rigidity in fact holds for a large class of conformal metrics that extend well beyond the known warped product examples. This generalizes the existing  results in the literature and produces new families of scalar–mean extremal and rigid manifolds.

\begin{theorem}\label{thm:comparison}
	Let $(N^n,\partial N,g)$ be a compact spin Riemannian manifold with boundary. Suppose that $u\in C^\infty(N)$ and $\widetilde g=u^2 g$ is a conformal metric. If
	\begin{enumerate}
		\item $N$ has nonzero Euler characteristic,
		\item $g$ has nonnegative curvature operator,
		\item $\partial N$ has nonnegative second fundamental form in $g$, and
		\item $u^{-1}$ has nonnegative Hessian with respect to $g$,
	\end{enumerate}
	then $(N,\partial N,\widetilde g)$ is spin scalar–mean extremal.
\end{theorem}

The assumption on nonvanishing Euler characteristic can be dropped in some cases. For example, the same result holds when $N$ has a torus factor (see Theorem~\ref{thm:comparisonTorus}). For instance, let $N=[0,T]\times\mathbb T^{n-1}$ with the flat metric $g=dt^2+g_{\textup{flat}}$. If we regard $\widetilde g=dt^2+\varphi(t)^2g_{\textup{flat}}$ as a conformal deformation of $g$, then condition (4) above is equivalent to the log–concavity of $\varphi$; see Example~\ref{example:warped}. In the special case where $g$ is flat and $u^{-1}$ has vanishing Hessian, we can drop the distance non-increasing assumption on $f$ in the interior of $M$. In particular,  Theorem~\ref{thm:comparison} and Theorem~\ref{thm:comparisonTorus} implies the scalar–mean comparison results of \cite{Cecchini:2021vs, WXDihedralWarped, ChaiWanWarped}. Our results apply to manifolds that are beyond log–concave warped products and provide new family of scalar-mean extremal or rigid metrics; see Section \ref{sec:examples}.

Our method of proof for Theorem~\ref{thm:comparison} is via Dirac operators. It remains unclear whether Theorem~\ref{thm:comparison} admits a non–spinorial proof; see for example the $\mu$-bubble approach for warped products \cite[Section 5.5]{Gromov4lectures2019} \cite{CWXZ}, and the capillary hypersurface approach for certain regions in Euclidean or hyperbolic spaces \cite{Lihyperbolic,WangWangZhu}.

With additional positivity assumptions on the  curvature of $N$ or the conformal factor $u$, we can upgrade the scalar-mean extremality result to a scalar-mean rigidity result.  

\begin{theorem}\label{thm:rigidity}
	With the same assumptions of  Theorem~\ref{thm:comparison}, if at every point $y\in N$, 
	\begin{enumerate}
		\item either Ricci curvature $\mathrm{Ric}(g)_y$ is positive, or
		\item the Hessian of $u^{-1}$ at $y$ is positive,
	\end{enumerate}
	then $(N,\partial N,\widetilde g)$ is spin scalar–mean rigid.
\end{theorem}

	This paper is organized as follows. In Section \ref{sec:extremality}, we prove Theorem \ref{thm:comparison} for the extremality theorem. In Section \ref{sec:rigidity}, we prove Theorem \ref{thm:rigidity} for the rigidity theorem. 
In Section \ref{sec:examples}, we give examples where Theorem \ref{thm:comparison} holds. In particular, Example \ref{example:warped} and \ref{example:warpedRegion} show that our results cover the known scalar-mean comparison results for log-concave warped products. We also present new family of spin scalar-mean extremal metrics.

	\section{Scalar-mean extremality}\label{sec:extremality}
In this section, we prove Theorem \ref{thm:comparison}.
\begin{proof}[Proof of Theorem \ref{thm:comparison}]
	Consider a compact Riemannian manifold $(M^n,\partial M,g_M)$ and a spin distance non-increasing map $\widetilde f\colon (M,g_M)\to (N,\widetilde g)$ with non-zero degree, such that $\Sc_{g_M}\geq \widetilde f^*\Sc_{\widetilde g}$ and $H_{g_M}\geq \widetilde f^*H_{\widetilde g}$.
	
	Set $u=e^\varphi$. The scalar curvature of the conformal metric $\widetilde g$ is given by
\begin{equation}\label{eq:conformalSc}
		\Sc_{\widetilde g}=e^{-2\varphi}(\Sc_{g}-2(n-1)\Delta_g \varphi-(n-1)(n-2)|d\varphi|^2)
\end{equation}
	Set $\psi=u^{-1}=e^{-\varphi}$. Then $d\psi=d(e^{-\varphi})=-e^{-\varphi}d\varphi$.
	
	Let $f\colon (M,g_M)\to (N,g)$ be the same map as $\widetilde f$, but viewed with respect to the background metric $g$. Since $\widetilde f$ is distance non-increasing, we have 
	$$\|f_*\|_x\leq u^{-1}(f(x))=e^{-\varphi(f(x))} \quad \textup{for all } x\in M.$$
	where $f_*\colon TM \to TN$ is the differential of $f$. 
	The bundle $TM\oplus f^\ast TN$ admits a spin structure, since $f$ is spin. Here  $TN$ denotes the tangent bundle over $N$ with respect to the metric $g$. Consider the spinor bundle $E\coloneqq S(TM\oplus f^*TN)$ of $TM\oplus f^*TN$ over $M$.  Let $\overbar c$ and $\hat c$ denote the Clifford actions of $TM$ and $f^*TN$ on $E$, respectively, and $\mathscr E$ the natural $\Z_2$-grading on $E$. Define $c(v)\coloneqq i\mathscr E \hat c(v)$ for all $v\in TN$,  so that $c$ and $\overbar c$ are commuting Clifford actions.
	
	Let $\nabla$ be the spinorial connection on $E$ and $D$ the associated Dirac operator
	$$D=\sum_{i=1}^n\overbar c(\overbar e_i)\nabla_{e_i}, $$
where  $\{\overbar e_i\}$ is a local orthonormal basis of $TM$.
	Consider the following Dirac operator with a potential 
	$$B=D+\Psi,~\textup{where }\Psi=\frac n 2\mathscr E c(d\psi).$$
	For any smooth section $\sigma$ of $E$, we have
	\begin{equation}\label{eq:B1}
		\|B\sigma\|^2=\int_M|D\sigma|^2+\frac{n^2}{4}|d\psi|^2|\sigma|^2+\langle [D,\Psi]\sigma,\sigma\rangle+\int_{\partial M}\langle\overbar c(\overbar\nu)\Psi\sigma,\sigma\rangle,
	\end{equation}
	where $\overbar\nu$ is the unit  inner normal vector of $\partial M$.
	Here we have used the fact that
	$$\int_M\langle D\sigma,\Psi\sigma\rangle+\langle \Psi\sigma,D\sigma\rangle=\int_M\langle [D,\Psi]\sigma,\sigma\rangle+\int_{\partial M}\langle\overbar c(\overbar\nu)\Psi\sigma,\sigma\rangle.$$
	Note that
	\begin{equation}\label{eq:commutator}
		[D,\Psi]=\frac n 2\sum_{i=1}^n\overbar c( \overbar e_i)\mathscr E c(\nabla_{f_* \overbar e_i}(d\psi)).
	\end{equation}

	The Hessian of $\psi=u^{-1}$ is given by
	$$\textup{Hess}_{\psi}(X,Y)=\langle\nabla_X(d\psi),Y\rangle, ~ \textup{ for all } X,Y\in TN.$$
	We define the Hessian operator $\mathbf{H}_\psi$ on $TN$ by
	\begin{equation}\label{eq:hessian}
		\langle \mathbf{H}_{\psi}(X),Y\rangle\coloneqq \textup{Hess}_{\psi}(X,Y),~ \textup{ for all } X,Y\in\Gamma(TN).
	\end{equation}
	Therefore, we obtain from \eqref{eq:B1} that
	\begin{equation}\label{eq:B2}
		\begin{split}
			\|B\sigma\|^2=&\int_M|D\sigma|^2+\frac{n^2}{4}|d\psi|^2|\sigma|^2+\frac{n}{2}\sum_i\langle \overbar c(\overbar e_i)\mathscr Ec(\mathbf{H}_{\psi}( f_*\overbar e_i))\sigma,\sigma\rangle\\
			&+\int_{\partial M}\frac{n}{2}\langle\overbar c(\overbar\nu)\mathscr E c(d\psi)\sigma,\sigma\rangle.
		\end{split}
	\end{equation}

	By the Stokes formula, we have
	\begin{equation}\label{eq:Stokes1}
		\int_M|D\sigma|^2=\int_M\langle D^2\sigma,\sigma\rangle+\int_{\partial M}\langle\overbar c(\overbar\nu) D\sigma,\sigma\rangle.
	\end{equation}
	By the Lichnerowicz formula, we have
	$$D^2=-\sum_{i=1}^n\nabla_{\overbar e_i}\nabla_{\overbar e_i}+\mathcal R,$$
	where $\mathcal R$ is the curvature tensor of $E$ given by 
	\begin{equation}\label{eq:R}
		\mathcal R=\frac{\Sc_{g_M}}{4}-\frac 1 4\sum_{i\ne j}\overbar c(\overbar e_i)\overbar c(\overbar e_j)c(R_g(f_*\overbar e_i\wedge f_*\overbar e_j))
	\end{equation}
	with $R_g$ the curvature operator of $(N,g)$. By \cite[Section 1.b]{GoetteSemmelmann}, we have
	\begin{equation}\label{eq:R>=}
		\mathcal R\geq \frac{\Sc_{g_M}}{4}-\frac{\Sc_{g}}{4 e^{2\varphi}}
	\end{equation}
	as $(N,g)$ has non-negative curvature operator and $\|f_*\|\leq e^{-\varphi}$.
	
	Consider a new connection on $E$ given by
	\begin{equation}\label{eq:conn}
		\hat\nabla_X=\nabla_X-\frac 1 2\overbar c(X)\mathscr E c(d\psi).
	\end{equation}
	We see that
	\begin{align*}
		\sum_{i=1}^n\hat\nabla_{\overbar e_i}^*\hat\nabla_{\overbar e_i}=&~\sum_{i=1}^n\left(-\nabla_{\overbar e_i}-\frac 1 2\overbar c(\overbar e_i)\mathscr E c(d\psi)\right)\left(\nabla_{\overbar e_i}-\frac 1 2\overbar c(\overbar e_i)\mathscr E c(d\psi)\right)\\
		=&~-\sum_{i=1}^n\nabla_{\overbar e_i}\nabla_{\overbar e_i}+\frac n 4|d\psi|^2+\frac 1 2\sum_i[\nabla_{\overbar e_i},\overbar c(\overbar e_i)\mathscr E c(d\psi)]\\
		=&~D^2-\mathcal R+\frac n 4|d\psi|^2+\frac 1 2\sum_{i=1}^n \overbar c(\overbar e_i)\mathscr Ec(\nabla_{f_*\overbar e_i}d\psi)\\
		=&~D^2-\mathcal R+\frac n 4|d\psi|^2+\frac 1 2\sum_{i=1}^n \overbar c(\overbar e_i)\mathscr Ec(\mathbf H_{\psi}(f_*\overbar e_i)).
	\end{align*}
	Therefore,
	\begin{equation}\label{eq:D^2Fredrich}
		D^2=\sum_i\hat\nabla_{\overbar e_i}^*\hat\nabla_{\overbar e_i}-\frac n 4|d\psi|^2-\frac 1 2\sum_i \overbar c(\overbar e_i)\mathscr Ec(\mathbf H_{\psi}(f_*\overbar e_i)) + \mathcal R.
	\end{equation}
	By the Stokes formula,
	\begin{equation}\label{eq:Stokes}
		\int_M\sum_i\langle\hat\nabla_{\overbar e_i}^*\hat\nabla_{\overbar e_i}\sigma,\sigma\rangle=\int_M|\hat\nabla\sigma|^2+\int_{\partial M}\langle\hat\nabla_{\overbar\nu} \sigma,\sigma\rangle.
	\end{equation}
	
	Now we consider the boundary terms from \eqref{eq:B2}, \eqref{eq:Stokes1} and \eqref{eq:Stokes}. Assume that $\sigma$ satisfies the following \emph{boundary condition}
	\begin{equation}\label{eq:bdryCondition}
		\mathscr E\overbar c(\overbar \nu)c(\nu)\sigma=-\sigma,
	\end{equation}
	where $\nu$ is the unit inner normal vector of $\partial N$ in $(N,g)$. As $\widetilde g$ is conformal to $g$, $\nu$ is also the inner normal vector of $\partial N$ in $(N,\widetilde g)$. We define 
	\begin{equation*}
		\overbar c_\partial(\overbar e_\lambda)\coloneqq\overbar c(\overbar \nu )\overbar c(\overbar e_\lambda) \textup{ and } c_\partial( e_\lambda)\coloneqq c(\nu) c( e_\lambda)
	\end{equation*}
	for $\overbar e_\lambda \in T(\partial M)$ and $e_\lambda \in T(\partial N)$. Moreover, we  define the following boundary connection on $E = S(TM\oplus f^\ast TN)$ over $\partial M$:
	\begin{equation*}
		\nabla^{\partial}_{\overbar e_\lambda}=\nabla_{\overbar e_\lambda}+\frac 1 2 \overbar c(\prescript{M}{}\nabla_{\overbar e_\lambda} \overbar \nu)\overbar c(\overbar \nu)  + \frac 1 2 c(\prescript{N}{}\nabla_{f_\ast\overbar e_\lambda}  \nu )c(\nu),
	\end{equation*}
	where $\prescript{M}{}\nabla$ and $\prescript{N}{}\nabla$ are Levi-Civita connections on $(M, g_M)$ and $(N, g)$ respectively.  We define  the boundary Dirac operator $D^\partial$  to be 
	\begin{equation}\label{eq:boundarydirac}
		D^{\partial}\coloneqq \sum_{j=1}^{n-1}\overbar c_\partial(\overbar e_j)\nabla^{\partial}_{\overbar e_j}.
	\end{equation}
	A straightforward computation (cf. \cite[Section 2]{Wang:2021tq}) shows that 
	\[  c(\overbar \nu) D + \nabla_{\overbar \nu} = D^\partial + \mathcal H,     \]
	where $\mathcal H$ is given by 	\begin{equation}\label{eq:H}
		\mathcal H=\frac{H_{g_M}}{2}-\sum_{\mu=1}^{n-1}\frac 1 2\overbar c(\overbar\nu)\overbar c(\overbar e_\mu)c(\nu)c(Af_*\overbar e_\mu),
	\end{equation}
	with $\{e_\mu\}$ a local orthonormal basis of tangent bundle $T(\partial M)$ and $A$ the second fundamental form of $\partial N$ in $(N,g)$.
	 In particular, the boundary terms from \eqref{eq:B2}, \eqref{eq:Stokes1} and \eqref{eq:Stokes} together give the following
	\begin{equation}\label{eq:boundary1}
		\begin{split}
		&\int_{\partial M}\langle\overbar c(\overbar \nu) D\sigma,\sigma\rangle+\langle\hat\nabla_\nu \sigma,\sigma\rangle+\frac{n}{2}\langle\overbar c(\overbar\nu)\mathscr E c(d\psi)\sigma,\sigma\rangle\\
		=&\int_{\partial M}\sum_{\mu=1}^{n-1}\langle\overbar c( \overbar\nu)\overbar c(e_\mu)\nabla_{e_\mu}\sigma,\sigma\rangle-\frac 1 2\langle\overbar c(\overbar\nu)\mathscr E c(d\psi)\sigma,\sigma\rangle+\frac{n}{2}\langle\overbar c(\overbar\nu)\mathscr E c(d\psi)\sigma,\sigma\rangle\\
		=&\int_{\partial M}\langle D^\partial\sigma,\sigma\rangle+\langle\mathcal H\sigma,\sigma\rangle+\frac{n-1}{2}\langle\overbar c(\nu)\mathscr E c(d\psi)\sigma,\sigma\rangle.
	\end{split}
	\end{equation}
	
	Now suppose that $\sigma$ satisfies the boundary condition \eqref{eq:bdryCondition}. It follows that  	$$\langle D^\partial\sigma,\sigma\rangle=0,$$
	and
	$$\frac{n-1}{2}\langle\overbar c(\overbar\nu)\mathscr E c(d\psi)\sigma,\sigma\rangle=\frac{n-1}{2}\langle d\psi,\nu \rangle|\sigma|^2.$$
	By \cite{Lottboundary} and \cite[Lemma 2.3]{Wang:2021tq}, we obtain that
	\begin{equation}\label{eq:H>=}
		\mathcal H\geq \frac{H_{g_M}}{2}-\frac{H_g}{2e^\varphi}
	\end{equation}
	as $\partial N$ in $(N,g)$ has non-negative second fundamental form and $\|f_*\|\leq e^{-\varphi}$.
	Thus the boundary contribution \eqref{eq:boundary1} becomes
	\begin{equation}\label{eq:bdry}
		\begin{split}
		&\int_{\partial M}\langle\overbar c(\nu) D\sigma,\sigma\rangle+\langle\hat\nabla_\nu \sigma,\sigma\rangle+\frac{n}{2}\langle\overbar c(\overbar\nu)\mathscr E c(d\psi)\sigma,\sigma\rangle\\
		\geq &\int_{\partial M}\left(\frac{ H_{g_M}}{2}-\frac{H_g}{2e^\varphi}+\frac{(n-1)\langle d\varphi,\nu\rangle}{2e^\varphi}\right)|\sigma|^2=\int_{\partial M}\left(\frac{H_{g_M}}{2}-\frac{H_{\widetilde g}}{2}\right)|\sigma|^2,
	\end{split}
	\end{equation}
	where we have used the fact that 
	\begin{equation}\label{eq:Hconformal}
		H_{\widetilde g}=e^{-\varphi}(H_{g}-(n-1)\langle d\varphi,\nu\rangle).
	\end{equation}
	
	To summarize, we have obtain that
\begin{equation}\label{eq:B3}
		\begin{split}
		\|B\sigma\|^2\geq& \int_M|\hat\nabla\sigma|^2+\frac{ \Sc_{g_M}}{4}|\sigma|^2-\frac{\Sc_g}{4\varphi^2}|\sigma|^2
		+(\frac{n^2}{4}-\frac{n}{4})|d\psi|^2|\sigma|^2\\
		&+\frac{n-1}{2}\sum_{i=1}^n\langle \overbar c(\overbar e_i)\mathscr Ec(\textbf{H}_{\psi}(f_*\overbar e_i))\sigma,\sigma\rangle\\
		&+\frac 1 2\int_{\partial M}(H_{g_M}-H_{\widetilde g})|\sigma|^2.
	\end{split}
\end{equation}

Recall that $\psi = e^{-\varphi}$. 	We claim that
	\begin{equation}\label{eq:>=Hess}
		\sum_{i=1}^n\langle \overbar c(\overbar e_i)\mathscr Ec(\textbf{H}_{\psi}(f_*\overbar e_i))\sigma,\sigma\rangle\geq \frac{-\Delta\psi}{e^\varphi}|\sigma|^2=\frac{-|d\varphi|^2+\Delta\varphi}{e^{2\varphi}}|\sigma|^2.
	\end{equation}
	In fact, we consider the singular value decomposition of the operator 
	\[ \mathbf{H}_\psi\circ f_*\colon TM\to TN, \] that is, consider locally orthonormal basis $\{\overbar u_i\}$ of $TM$ and $\{v_i\}$ of $TN$ such that
	$$\textbf{H}_\psi(f_*\overbar u_i)=\lambda_iv_i$$
	for some  $\lambda_i\geq 0$. 
	Therefore, we have
	\begin{equation}\label{eq:traceNorm}
		\sum_{i=1}^n\overbar c(\overbar e_i)\mathscr Ec(\textbf{H}_{\psi}(f_*\overbar e_i))=\sum_{i=1}^n \overbar c(\overbar u_i)\mathscr Ec(v_i)\lambda_i\geq -\sum_{i=1}^n\lambda_i=-\|\textbf{H}_\psi\circ f_*\|_1,
	\end{equation}
	where $\|\cdot\|_1$ denotes the trace norm. By the H\"older inequality, we have
	\begin{equation}\label{Holder}
		\|\textbf{H}_\psi\circ f_*\|_1\leq \|\textbf{H}_\psi\|_1\cdot\|f_*\|\leq\tr(\textbf{H}_\psi) e^{-\varphi}=\frac{\Delta\psi}{e^\varphi}, 
	\end{equation}
	where we used the assumption that the Hessian of $\psi = u^{-1}$ is nonnegative in the second inequality.  This finishes the proof of the claim.
	
	Therefore, line \eqref{eq:>=Hess} together with the equation  \eqref{eq:conformalSc} implies that 
	\begin{equation}\label{eq:>=conformalSc}
		\begin{split}
		&-\frac{\Sc_{g}}{4e^{2\varphi}}
		+(\frac{n^2}{4}-\frac{n}{4})|d\psi|^2+\frac{n-1}{2}\sum_{i=1}^n\overbar c(e_i)\mathscr Ec(\textbf{H}_{\psi}(f_*e_i))\\
		\geq &-\frac{\Sc_{g}}{4e^{2\varphi}}
		+\frac{n(n-1)}{4}\frac{|d\varphi|^2}{e^{2\varphi}}+\frac{n-1}{2}\cdot\frac{-|d\varphi|^2+\Delta\varphi}{e^{2\varphi}}\\
		= &-\frac 1 4\left(\frac{\Sc_g}{e^{2\varphi}}-n(n-1)\frac{|d\varphi|^2}{e^{2\varphi}}-2(n-1)\frac{-|d\varphi|^2+\Delta\varphi}{e^{2\varphi}}\right)\\
		=&-\frac 1 4e^{-2\varphi}\left(\Sc_g-2(n-1)\Delta\varphi-(n-1)(n-2)|d\varphi|^2\right)=-\frac 1 4\Sc_{\widetilde g}.
	\end{split}
	\end{equation}
	
To summarize,  we have 	\begin{equation}\label{eq:B4}
		\|B\sigma\|^2\geq\int_M|\hat\nabla\sigma|^2+\frac{\Sc_{g_M}-\Sc_{\widetilde g}}{4}|\sigma|^2+\int_{\partial M}\frac{H_{g_M}-H_{\widetilde g}}{2}|\sigma|^2.
	\end{equation}
By assumption, $\Sc_{g_M}\geq \widetilde f^*\Sc_{\widetilde g}$ and $H_{g_M}\geq \widetilde f^*H_{\widetilde g}$. Therefore, if either $\Sc_{g_M}\ne \widetilde f^*\Sc_{\widetilde g}$ somewhere on $M$ or $H_{g_M}\ne \widetilde f^*H_{\widetilde g}$ somewhere on $\partial M$, then  the operator $B$ subject to the boundary condition \eqref{eq:bdryCondition} is invertible. However, by the Atiyah--Singer index theorem, we have
	$$\ind(B)=\ind(D)=\deg(f)\chi(N)\ne 0.$$
	This leads to a contradiction and finishes the proof.
\end{proof}

\section{Scalar-mean rigidity}\label{sec:rigidity}
We now prove Theorem \ref{thm:rigidity} in this section.

\begin{proof}[Proof of Theorem \ref{thm:rigidity}]
	Assume that $(M^n,\partial M,g_M)$ is a compact Riemannian manifold with boundary and $\widetilde f\colon (M,g_M)\to (N,\widetilde g)$ is a spin distance non-increasing map with non-zero degree such that $\Sc_{g_M}= \widetilde f^*\Sc_{\widetilde g}$ and $H_{g_M}= \widetilde f^*H_{\widetilde g}$. We use the same notation as in the proof of Theorem \ref{thm:comparison}. 
	
	The proof is to track the inequalities in the proof of Theorem \ref{thm:comparison}. By the non-vanishing of the index of $B$, there exists a non-zero section $\sigma$ of $E$ satisfying the boundary condition \eqref{eq:bdryCondition} and $B\sigma=0$. It follows from \eqref{eq:B4} that $\hat\nabla\sigma=0$, so $\sigma$ is nowhere vanishing on $M$. Therefore, various  inequalities in the proof of Theorem \ref{thm:comparison} must, in fact, be  equalities on $M$.
	
	Let us first consider the inequality \eqref{eq:R>=}. For any $x\in M$, take the singular value decomposition of $R_g\circ \wedge^2 f_*\colon \Bigwedge^2T_xM\to \Bigwedge^2T_{f(x)}N$. There exist orthonormal basis $\{\overbar  \alpha_k\}$ of $\Bigwedge^2T_xM$, $\{\beta_k\}$ of $\Bigwedge^2T_{f(x)}N$, and $\lambda_i\geq 0$, such that
	$$R_g\circ \wedge^2 f_*(\overbar \alpha_k)=\lambda_k\beta_k.$$
	Therefore, we have
	\begin{align*}
		&\frac 1 4\sum_{i\ne j}\overbar c(\overbar e_i)\overbar c(\overbar e_j)c(R_g(f_*\overbar e_i\wedge f_*\overbar e_j))
		=\frac 1 2\sum_{k=1}^{n(n-1)/2}\overbar c(\overbar \alpha_k)c(\beta_k)\lambda_k\\
		\leq~&\frac 1 2\sum_{k=1}^{n(n-1)/2}\lambda_k	=\frac 1 2\|R_g\circ \wedge^2 f_*\|_1\\
		\leq~&\frac 1 2\tr(R_g)\cdot\|f_*\|^2\leq \frac 1 2\tr(R_g)\cdot\frac{1}{e^{2\varphi}}=\frac{\Sc_g}{4e^{2\varphi}}.
	\end{align*}
	Since the equality  $\Sc_{g_M}= \widetilde f^*\Sc_{\widetilde g}$ is attained, every inequality in the above steps must be an equality. In particular, we have 
	\begin{equation}\label{eq:holderequal}
		\frac 1 2\|R_g\circ \wedge^2 f_*\|_1
		=   \frac 1 2\tr(R_g)\cdot\|f_*\|^2.
	\end{equation}  Since $R_g\geq 0$, $R_g$ admits a square root $\sqrt{R_g}$. Therefore, we have
	\begin{align*}
		&\|R_g\circ \wedge^2 f_*\|_1=\|\sqrt{R_g}\cdot\sqrt{R_g}\circ\wedge^2 f_*\|_1\\
		\leq~&\|\sqrt{R_g}\|_2\cdot\|\sqrt{R_g}\circ\wedge^2 f_*\|_2=\sqrt{\tr(R_g)} \cdot  \sqrt{\tr((\wedge^2 f_*)^*\circ R_g\circ \wedge^2 f_*)}\\
		\leq~&\sqrt{\tr(R_g)\cdot\tr(R_g)\|f_*\|^4}=\tr(R_g)\cdot \|f_*\|^2=\frac 1 2\Sc_g\cdot \|f_*\|^2,
	\end{align*}
	which must be an equality because of line \eqref{eq:holderequal}. Let us focus on the consequence of the inequality 
	\[  \tr((\wedge^2 f_*)^*\circ R_g\circ \wedge^2 f_*) \leq \tr(R_g)\|f_*\|^4 = \frac 1 2\Sc_g\cdot \|f_*\|^4 \]
	being an equality. Consider the singular value decomposition of $f_*$, namely an orthonormal basis $\{\overbar u_i\}$ of $T_xM$, an orthonormal basis of $\{v_i\}$ of $T_{f(x)}N$ and $\rho_i\in[0,1]$, such that
	$$f_* \overbar u_i=\rho_i v_i,~\forall i.$$
	Then we have
	\begin{align*}
		~&\tr((\wedge^2 f_*)^*\circ R_g\circ \wedge^2 f_*)\\
		=~&\sum_{i<j}
		\langle(\wedge^2 f_*)^*\circ R_g\circ \wedge^2 f_*) \overbar u_i\wedge \overbar u_j, \overbar u_i\wedge \overbar u_j\rangle
		=\sum_{i<j}
		\langle R_g(f_*\overbar u_i\wedge f_*\overbar u_j),f_*\overbar u_i\wedge f_*\overbar u_j\rangle\\
		=~&\sum_{i<j}\rho_i^2\rho_j^2\langle R_g(v_i\wedge v_j),v_i\wedge v_j\rangle\\
		=~&\sum_{i<j}\rho_i^2\rho_j^2 K_g(v_i,v_j)\leq \|f_*\|^4\sum_{i<j} K _g(v_i,v_j)=\frac{1}{2}\|f_*\|^4 \cdot \Sc_g,
	\end{align*}
	where $K_g$ stands for the sectional curvature of $g$. It follows from the above discussion that 
	\[  \sum_{i<j}\rho_i^2\rho_j^2 K_g(v_i,v_j) =  \|f_*\|^4\sum_{i<j} K _g(v_i,v_j).\] This implies that
	$$\rho_i^2\rho_j^2 K_g(v_i,v_j)= \|f_*\|^4K_g(v_i,v_j),~\forall i<j.$$
	Therefore, if assumption (1) holds, namely,  $(N,g)$ has positive Ricci curvature at $x$, then for every $i$, there exists $j\neq i$ such that $K_g(v_i, v_j) \neq 0$, hence  $\rho_i \rho_j = \|f_\ast\|^2$. It follows that  $\rho_i = \|f_\ast\|$ for all $1\leq i \leq n$. Furthermore, since 
	\[ \frac 1 2\tr(R_g)\cdot\|f_*\|^2\leq \frac 1 2\tr(R_g)\cdot \frac{1}{e^{2\varphi}} \] also has to be an equality, it follows that $\|f_\ast\| = \frac{1}{e^{2\varphi}}$. In other words, all singular eigenvalues of $\widetilde f_\ast$ are equal to $1$, hence $\widetilde f_\ast \colon (T_xM, g_M) \to (T_{f(x)}N, \widetilde g)$ is an isometry.

	If instead assumption (2) holds, that is, the Hessian of $\psi = u^{-1}$ is positive at $x$,  then the H\"older inequality \eqref{Holder} used in the proof of Theorem \ref{thm:comparison} must be an equality. Now a completely similar  argument as above shows that $\widetilde f_\ast \colon (T_xM, g_M) \to (T_{f(x)}N, \widetilde g)$ is an isometry. 
	
	Since $x$ is arbitrary in $M$, it follows that  $\widetilde f\colon (M, g_M) \to (N, \widetilde g)$ is Riemannian covering map. This finishes the proof.

\end{proof}

\section{Examples of the model conformal metrics}\label{sec:examples}

In this section, we give several  concrete examples of the model space $(N^n,\partial N,\widetilde g)$ appearing in Theorem \ref{thm:comparison}.

Firstly,  the assumption on nonvanishing Euler characteristic in Theorem \ref{thm:comparison}  can be dropped in some cases. For example, we have the following theorem. 
	\begin{theorem}\label{thm:comparisonTorus}
	Let $(N^n,\partial N,\widetilde g)$ be a compact spin Riemannian manifold with boundary. Suppose that $g$ is a metric on $N$ such that $\widetilde g=u^2 g$ for some positive function $u\in C^\infty(N)$. If
	\begin{enumerate}
		\item $N=X^k\times\mathbb T^{n-k}$, where $X$ is  a manifold with non-zero Euler characteristic and $\mathbb T^{n-k}$ is the $(n-k)$-torus,
		\item $g=g_X+g_{\textup{flat}}$, where $g_X$ has non-negative curvature operator and $g_{\textup{flat}}$ is a flat metric on the torus,
		\item $\partial N=\partial X\times\mathbb T^{n-k}$ has non-negative second fundamental form in $g$,
		\item $u^{-1}$ has non-negative Hessian with respect to $g$,
	\end{enumerate}
	then $(N,\partial N,\widetilde g)$ is spin scalar-mean extremal.
\end{theorem}
\begin{proof}
	The proof essentially follows the exact same strategy as that of Theorem \ref{thm:comparison}, except we shall consider a large finite cover of $N$ and twist the Dirac operator by an extra almost flat Bott vector bundle along the torus. This extra almost flat bundle makes the index of the  corresponding twisted Dirac operator have non-vanishing index. Of course, this extra bundle will introduce small error terms in  the corresponding estimates analogous to those in the proof of Theorem \ref{thm:comparison}. However, these small error terms can be handled by a Poincar\'e type  inequality as in \cite[proof of Theorem 1.5 in Section 4]{WangXiedegnerate}.
\end{proof}

Theorem \ref{thm:comparison} and Theorem \ref{thm:comparisonTorus} generalize various  scalar-mean extremality and rigidity results of warped product manifolds or submanifolds of warped product manifolds in the literature. We give two examples here. 

\begin{example}\label{example:warped}
	Let $Y^{n-1}=Z^{k}\times\mathbb T^{n-1-k}$ be the product of a closed manifold with non-zero Euler characteristic and the $(n-1-k)$-torus, equipped with a metric $h$ with non-negative curvature operator. Suppose that $N=[0,1]\times Y$ is equipped with the warped product metric
	$$\widetilde g=dt^2+\varphi(t)^2h.$$
	Denote	$g=dt^2+h$, which has non-negative curvature operator. Note that $\widetilde g$ can  also be viewed as  a conformal change of $g$. In fact, if we set $t=\rho(s)$, then $$\widetilde g=(\rho')^2ds^2+\varphi(\rho(s))^2h.$$
	If $\rho$ solves the differential equation $\rho'=\varphi(\rho)$, then we may choose  $u=\rho'=\varphi(\rho)$ to be the conformal factor. In this case, we have  
	$$du^{-1}=-\frac{\varphi'(\rho(s))\rho'(s)}{\varphi(\rho(s))^2}ds=-\frac{\varphi'(\rho(s))}{\varphi(\rho(s))}ds.$$
	Therefore, the non-negativity of the Hessian of $u^{-1}$ is equivalent to that
	$$\left(\frac{\varphi'}{\varphi}\right)'\cdot\rho'\leq 0.$$
	As $\rho'=\varphi(\rho)>0$, this is equivalent to 
	\[  \left(\frac{\varphi'}{\varphi}\right)'\leq 0, \]
	that is,  $\varphi$ is log-concave. In particular, Theorem \ref{thm:comparisonTorus} recover the scalar-mean extremality and rigidity theorems of \cite{Cecchini:2021vs}.
\end{example}
\begin{example}\label{example:warpedRegion}
	We follow the same notation of Example \ref{example:warped}. Now suppose $X$ is a compact region with boundary in $[0,1]\times Y$, which is equipped with a warped product metric $\widetilde g=dt^2+\varphi(t)^2h$. Condition (4) in Theorem \ref{thm:comparison} is equivalent to that $\varphi$ is log-concave, as shown in Example \ref{example:warped} above. Condition (3) now yields that $\partial X$ has non-negative second fundamental form in the direct product metric $g=dt^2+h$. In particular, Theorem \ref{thm:comparison}  and Theorem \ref{thm:comparisonTorus} recover the scalar-mean extremality and rigidity theorems of \cite{WXDihedralWarped,ChaiWanWarped}.
\end{example}

Theorem \ref{thm:comparisonTorus} exhibit new scalar-mean extremal examples of warped product metric over torus.
\begin{example}\label{example:warpTorus}
	Let $(Y^m,g_Y)$ be a manifold with boundary, and set $N^n=Y^m\times\mathbb T^{n-m}$. For a smooth function $u$ on $Y$, consider the warped product metric on $N$ given by
	$$\widetilde g=g_Y+u^2 g_{\text{flat}},$$
	which is a conformal deformation of
	$$g=u^{-2}g_Y+g_{\textup{flat}}$$
	with conformal factor $u$. By Theorem \ref{thm:comparison}, if
	\begin{enumerate}
		\item $Y$ has non-zero Euler characteristic,
		\item the curvature operator of $u^{-2}g_Y$ is non-negative,
		\item $\partial Y$ has non-negative second fundamental form with respect to $u^{-2}g_Y$,
		\item $u^{-1}$ has non-negative Hessian with respect to $u^{-2}g_Y$,
	\end{enumerate}
	then $(Y\times\mathbb T^{n-m},\partial Y\times \mathbb T^{n-m},g_Y+u^2 g_{\text{flat}})$ is scalar-mean extremal.
	
Now assume that $Y^m$ is a ball and $g_Y=dr^2+\rho(r)^2g_{\sph^{m-1}}$ is a smooth warped product metric for $r\in[0,R]$. If $u=u(r)$ is radial, then
	$$\widetilde g=dr^2+\rho(r)^2g_{\sph^{m-1}}+u(r)^2g_{\text{flat}}$$
	is a doubly warped product metric.
	
	In particular, for $m=2$, the metric becomes
	$$\widetilde g=dr^2+\rho(r)^2d\theta^2+u(r)^2g_{\text{flat}}.$$
	Direct computation shows that
	$$K=-\frac{u}{\rho}\cdot \left(u\left(\frac{\rho}{u}\right)'\right)',$$
	where $K$ is the Gauss curvature of $u^{-2}(dr^2+\rho^2d\theta^2)$, and
	$$\textbf{H}_{u^{-1}}\left(\frac{\partial_r}{u}\right)=-(\log u)''u\cdot \frac{\partial_r}{u},~\textbf{H}_{u^{-1}}\left(\frac{\rho\partial_\theta}{u}\right)=-u'\left(\frac{\rho}{u}\right)'\cdot \frac{\rho\partial_\theta}{u}.$$
	Therefore, $(Y\times\mathbb T^{n-m},\sph^1\times \mathbb T^k,\widetilde g)$ is scalar-mean extremal if
	$$\left(\frac{\rho}{u}\right)'\Big|_{r=R}\geq 0,~(\log u)''\leq 0,~u'\left(\frac{\rho}{u}\right)'\leq 0,\text{ and }\left(u\left(\frac{\rho}{u}\right)'\right)'\leq 0.$$
	For example, we consider that $\rho(r)=\sin(r)$, and $u$ is any smooth function on $[0,R]\subset [0,\pi/2)$ such that $u\geq 1>0$, $u'(0)=0$, $u'\leq 0$, and $(\log u)''\leq 0$. Then
	$$\widetilde g=dr^2+\sin(r)^2d\theta^2+u(r)^{2\varepsilon}g_{\text{flat}}$$
	satisfies the above requirements provided that $\varepsilon$ is a sufficiently small positive number.
\end{example}

Theorem \ref{thm:comparison} also provides new families of scalar-mean extremal metrics on Euclidean balls, which are neither non-negatively curved or warped products.
\begin{example}\label{example:disk}
	Let $N=\mathbb D^n$ ($n\geq 2$) be the unit ball in $\R^n$ equipped with the standard Euclidean metric $g=g_{\textup{eu}}$. Let $\psi$ be smooth positive function on $\mathbb D^n$ with non-negative Hessian, i.e., $\psi$ is a convex function. Set $u=\psi^{-1}$. Then by Theorem \ref{thm:comparison}, 
	$$\widetilde{g}=u^2 g_{\textup{eu}}=\psi^{-2}g_{\textup{eu}}$$
	is scalar-mean extremal.
	
	As a concrete example, let us consider the convex function
	$$\psi=1+\sum_{i=1}^n a_i x_i^2$$
	where $a_i>0$. Then $\widetilde{g}=\psi^{-2}g_{\textup{eu}}$ is scalar-mean rigid. Its scalar curvature is given by 
	$$\Sc_{\widetilde g}=4(n-1)\left(S+\sum_{i=1}^n a_i x_i^2(S-n a_i)\right)$$
	with $S=\sum_{i=1}^n a_i$. When the $a_i$'s are distinct, the scalar curvature takes both positive and  negative values. Such metrics are not non-negatively curved or warped product metrics: the Ricci tensor of $\widetilde g$  has  distinct eigenvalues everywhere, and $\widetilde g$ does not admit any nontrivial  Killing vector field.
\end{example}

In light of Example \ref{example:disk}, we propose the following question.
\begin{question}
	Besides the standard Euclidean metrics and its conformal deformation as in Example \ref{example:disk}, is it possible to give a complete classification of scalar-mean extremal or rigid metrics on the Euclidean ball?
\end{question}

\end{document}